\documentclass[a4paper]{amsart}
\usepackage{amssymb}
\usepackage{verbatim}
\usepackage{amscd}

\def\ent{{{\operatorname{h}}}}
\def\Det{{{\operatorname{Det}}}}
\numberwithin{equation}{section}
\theoremstyle{plain}
\newtheorem{theorem}[equation]{Theorem}
\newtheorem{corollary}[equation]{Corollary}
\newtheorem{proposition}[equation]{Proposition}
\newtheorem{lemma}[equation]{Lemma}
\newtheorem*{(DQ1)}{(DQ1)}
\theoremstyle{definition}

\theoremstyle{remark}

\begin{document}
\title [Highly symmetric Markov-Dyck shifts]{On a class of  highly symmetric\\Markov-Dyck shifts}
\author{Wolfgang Krieger}
\begin{abstract}
A class of highly symmetric Markov-Dyck shifts is introduced. Topological entropies and zeta functions are determined.
\end{abstract}
\maketitle

\section{Introduction}
Let $\Sigma$ be a finite alphabet, and let $S_{\Sigma^{\Bbb Z}}$ be the left shift on
$\Sigma^{\Bbb Z}$,
$$
S_{\Sigma^{\Bbb Z}}((\sigma_i)_{ i \in {\Bbb Z}})   = (\sigma_{i+1})_{ i \in {\Bbb Z}}, \qquad  
(\sigma_i)_{ i \in {\Bbb Z}} \in \Sigma^{\Bbb Z}.
$$
We denote the restriction of $S_{\Sigma^{\Bbb Z}}$ to  a $S_{\Sigma^{\Bbb Z}}$-invariant Borel subset $X$ of $\Sigma^{\Bbb Z}$ by$S_X$.
Compact  $S_{\Sigma^{\Bbb Z}}$-invariant subsets $X$ of $\Sigma^{\Bbb Z}$, or rather the dynamical systems  $(X, S_{X})$, are called subshifts.
Transition matrices $A_{\sigma, \sigma^\prime}  \in \{0, 1\}, \sigma, \sigma^\prime \in \Sigma $, that have in every row and in every column at least one entry, that is equal to 1, define Markov shifts 
$$
M(A) = \bigcup_{i \in \Bbb Z}\{(\sigma_i)_{i \in \Bbb Z}\in \Sigma^{\Bbb Z}: A(\sigma_i,\sigma_{i + 1}) = 1\},
$$
that serve as
prototypical examples of subshifts. We denote the Perron eigenvalue of a $\Bbb Z_+$-matrix $A$ by $\lambda(A)$. The topological entropy of the Markov shift 
 with transition matrix $A$  is given by
\begin{align*}
\ent(M(A)) = \log \lambda. \tag{1.1}
\end{align*}
For an introduction to the theory of subshifts see \cite{Ki} and \cite{LM}. 
A finite word in the symbols of 
$\Sigma$  is called admissible for the subshift 
$X \subset \Sigma^{\Bbb Z} $ if it
appears in a point of $X$. 
A subshift $X \subset \Sigma^{\Bbb Z} $ is uniquely determined by its language 
of admissible words. 

In this paper we are concerned with subshifts that are constructed from directed graphs. 
We denote a finite directed graph with vertex set $\mathcal V$ and edge set $\mathcal E$ by 
$G(\mathcal V,\mathcal E)$. The source vertex of an edge $e \in {\mathcal E}$ 
we denote by
$s$ and its target vertex by $t$. A strongly connected directed graph 
$G = G(\mathcal V,\mathcal E)$ has an edge shift, which is a Markov shift with alphabet 
$\mathcal E$ and whose language of admissible words is the set of finite paths in $G$. We denote the edge shift of $G$ by  
$M\negthinspace{\scriptstyle E} (G).$

Another class of subshifts that arise from directed graphs, are the Markov-Dyck shifts.
We recall their construction.
For a given graph $G = G(\mathcal V,\mathcal E)$, let 
$
{\mathcal E}^- = \{e^-: e \in   \mathcal E  \}
$
be a copy of ${\mathcal E}$.
Reverse the directions of the edges in ${\mathcal E}^-$
to obtain the 
edge set
$
{\mathcal E}^+ = \{e^+: e \in   \mathcal E  \}
$
of the reversed graph 
of $G(\mathcal V  , \mathcal E^-)  $.
In this way one has defined a graph ${G}( \mathcal V  , {\mathcal E}^- \cup{\mathcal E}^+  )$, that has  source and target mappings, that are given by
\begin{align*}
&s(e^-) =s(e)     , \quad  t(e^-) = t(e)    ,
\\
&s(e^+) =t(e)  , \quad t(e^+) = s(e) , \qquad e \in   \mathcal E .
\end{align*}
The construction of the Markov-Dyck shift of a finite strongly connected non-circular  directed graph 
$G(\mathcal V  , \mathcal E^-)  $ is via the graph inverse semigroup of the graph 
\cite [Section 7.3]{L}.
With idempotents $\bold 1_V, V \in {\mathcal V},$ we use as a generating set of the graph inverse semigroup of $G(\mathcal V  , \mathcal E^-)  $
the set 
${\mathcal E}^- \cup \{\bold 1_V:V \in {\mathcal V}\}\cup {\mathcal E}^+$. 
Besides $ \bold 1_V^2= \bold 1_V, V \in {\mathcal V}$, we have the relations
$$
\bold 1_U\bold 1_W = 0, \qquad   U, W \in {\mathcal V}, U \neq W,  
$$
$$
f^-g^+ =
\begin{cases}
\bold 1_{s(f)}, &\text{if  $f = g$}, \\
0, &\text {if  $f \neq g$},\quad f,g \in {\mathcal E},
\end{cases}
$$
\begin{equation*}
e^-=\bold 1_{s(e)} e^- = e^- \bold 1_{t(e)},\qquad
\bold 1_{t(e)} e^+ = e^+ \bold 1_{s(e)},\qquad
e \in {\mathcal E}. 
\end{equation*}
The alphabet of  the Markov-Dyck shift of the strongly connected non-circular directed graph $G(\mathcal V  , \mathcal E))$ is ${\mathcal E}^-\cup {\mathcal E}^+$, and
its admissible words are the words 
$$
(e_i)_{1 \leq i \leq I} \in ({\mathcal E}^-\cup {\mathcal E}^+)^I, \qquad I \in \Bbb N,
$$
such that
$$
\prod_{1 \leq k \leq I}\sigma_i \neq 0.
$$
We denote the Markov-Dyck shift of the directed graph $G$ by 
$M\negthinspace {\scriptstyle D}({G})$.
Markov-Motzkin shifts \cite [Section 4.1]{KM2} are versions of Markov-Dyck shifts.

For a given graph $G = G(\mathcal V,\mathcal E)$ set for $e \in \mathcal E$
$$
(e^-)^{-1} = e^+, \quad (e^+)^{-1} = e^-.
$$
The map 
$
e \to (e)^{-1} (e \in \mathcal E)
$
induces  an anti-automorphism of the graph inverse semigroup of $G$, and also a time reversal $\rho$ of the Markov-Dyck shift of $G$, by
$$
(\rho  (x))_i = (x_{-i})^{-1}, \qquad x \in 
M\negthinspace {\scriptstyle D}(\widehat{G})
$$
We refer to $\rho$ as the canonical time reversal of he Markov-Dyck shift of $G$.

Denote the one-vertex directed graph with $N>1$ loops by $G(N)$. The graph inverse semigroup of $G(N)$ is the Dyck inverse monoid (the "polycyclique" \cite{NP}), and the subshifts 
$M\negthinspace {\scriptstyle D}(\widehat{G}(N))$
 are the Dyck shifts, that were 
introduced in
\cite {Kr1}.
As shown in 
\cite [Section 4]{Kr1} 
the topological entropy of $M\negthinspace {\scriptstyle D}({G})$ is given by
\begin{align*}
\ent(M\negthinspace {\scriptstyle D}({G}))  = \log (N + 1). \tag {1.2}
\end{align*}
It was shown in 
\cite [Section 4]{Kr1}, that the Dyck shifts have two measures of maximal entropy.
The zeta function of the Dyck shift was obtained in \cite[Example 3, p. 79] {Ke}.
Also necessary and sufficient conditions for the existence of an embedding of an irreducible subshift of finite type into a Dyck shift are known 
\cite {HI}.
K-groups of the Dyck shifts were computed in \cite{Ma1} and \cite {KM1}.
For the tail invariant measures of the Dyck shifts see \cite{Me}.

Another example of a Markov-Dyck shift is the Fibonacci-Dyck shift, that arises from the Fibonacci  graph $F$, that has two vertices $V_1$ and $V_2$, and three edges $e_{1,2} ,e_{2,1}$ and $e_{1,1}$, and source and target mappings, that are given by
$$
V_1 = s( e_{1,2} ) = t( e_{2,1} ) = s(e_{1,1} ) = t(e_{1,1} ), \qquad 
V_2 = s( e_{2,1 )}) = t( e_{1,2} ).
$$
In 
 \cite[Section 4] {KM2}
 the topological entropy of $M\negthinspace {\scriptstyle D}(\widehat{F})$ was shown to be  given by
 $$
  \ent (M{\scriptstyle D} (G(N))) = 3\log 2 - \log 3,
  $$
and a formula for the zeta function of the Fibonacci-Dyck shift was also given.
Necessary and sufficient conditions for the existence of an embedding of an irreducible subshift of finite type into the Fibonacci-Dyck shift are known 
\cite {HK}. 
K-groups of the Fibonacci-Dyck shift were computed in \cite{Ma4}

In this paper we introduce 
a class of directed graphs, that are built from directed trees, and we study  their Markov-Dyck shifts
We consider directed trees all of whose edges are  pointing away from the root. The height of  a vertex $V$ of the tree is the length of the path from the root to $V$. A directed tree is said to be rotationally homogeneous, if all of its leaves have the same height, and if vertices, that have the same height, have the same out-degree.
For a finite directed graph $G = G( \mathcal V , \mathcal E )$ 
 denote  by $\mathcal F_G$ the set of edges that are the only incoming edges of their target  vertices. We say that the directed graph $G$ is rotationally homogeneous, if 
 $\mathcal F_G$ is a rotationally homogeneous directed tree, if
 the source vertex of every edge $e \in  \mathcal E \setminus \mathcal F_G$ is a leaf of 
 $\mathcal F_G$, and if all leaves of $\mathcal F_G$ have the same out-degree in $G$.
 The height of a rotationally homogeneous directed tree   is the height of the leaves of  $\mathcal F_G$.
 
Topological entropy and zeta functions of Markov-Dyck shifts were studied in \cite {KM2}. In this paper we study topological entropy and zeta functions of the Markov-Dyck shifts of rotationally homogeneous directed graphs with emphasis on the special features, that appear due to the symmetry of the graphs. 
For the problem of topological conjugacy of  rotationally homogeneous directed graphs see \cite {HK}.
In section 2 we consider rotationally homogeneous directed graphs of height $H(G) > 1$ and  in Section 3 of height  $H(G)= 1$. 
The directed graphs $G(N), N > 1,$ can be interpreted as  rotationally homogeneous directed graphs of height zero. From this point of view  
Sections 2 and 3  of this paper extend the results and  the methods of \cite [Section 4] {Kr1}. We show that the Markov-Dyck shifts of rotationally homogeneous directed graphs have two measures of maximal entropy. We also identify for a given rotationally homogeneous  directed graph $G$ a companion graph, that has  an edge shift with topological entropy equal to the topological entropy of the Markov-Dyck shift of $G$.
As in \cite [Section 4] {Kr1} this involves constructing 
suitable shift invariant Borel subsets of the Markov-Dyck shift of $G$ and of the edge shift of the companion graph, that yield Borel conjugate 
Borel dynamical systems. From there we obtain by simple means explicit formulas for the topological entropy of the Markov-Dyck shifts of rotationally homogeneous directed graphs of height 2, 3, 4, 5, and 7.

Given a directed graph $G =  G( \mathcal V , \mathcal E )$,
set
$$
\psi(e^-) = 1, \ \ \psi(e^+) = - 1, \qquad e \in \mathcal E.
$$
Following the terminology, that was introduced in \cite{HI}, we say that a point 
$x\in M\negthinspace {\scriptstyle D}({G})$
of period $\pi$ is neutral, if there exists an $i\in \Bbb Z$, such that 
$
\sum_{i\leq j < i + \pi} \psi(x_j) 
 $
is zero, and we say that it has a
negative (positive) multiplier, if there
exists an $i\in \Bbb Z$, such that 
$
\sum_{i\leq j < i + \pi} \psi(x_j) 
 $
is negative (positive).
General expressions for the zeta functions of Markov-Dyck shifts were derived in 
\cite [Section 2]{KM2}. Following \cite [Section 3]{Ke} these expressions were obtained by factoring the zeta function according to the classification of the periodic points of the Markov-Dyck shift as neutral, or as periodic points with negative or with positive multiplier, and by using for each case a suitable circular code. 
The method is based on the fact, that for circular codes  $\mathcal C$ in the symbols of a finite alphabet $\Sigma$, the zeta function of the set of periodic points in 
$\Sigma^\Bbb Z$, that carry a bi-infinite concatenation of code words, is related to the generating function $g_\mathcal C$ of $\mathcal C$ by
$$
\zeta_{X(\mathcal C)}= \frac{1}{1 - g_{\mathcal C}}.
$$
(see \cite[Section 7.3] {BPR}).
See also
 \cite [Section E] {BBD1},
 \cite [Section 3] {Ma2},
 \cite [Section 6] {BBD2},
 \cite {IK1},
 \cite [Section 2] {IK2},
 \cite [Section 5.2] {BH1},
 \cite [Section 4] {BH2}, and
 \cite [Section 9] {Kr2}.
In Section 4 we consider the zeta functions of Markov-Dyck shifts  of rotationally homogeneous directed graphs. We find that the symmetry of the graphs makes it possible to determine in  this case  the relevant generating functions.
 
\section{The case $H(G) > 1$}

We define sets of data by
$$
\bold D_H = \{(N_h)_{1\leq h \leq H + 1}: \prod_{1 \leq h \leq H + 1} N_h> 1\}, \quad 
H > 1. 
$$
A data $\bold N= (N_h)_{1\leq h \leq H + 1} \in \bold D_H, H > 1$, specifies a rotationally homogeneous directed graph 
$\widehat {G}(\bold N)$ with vertex set
$$
\mathcal V(\bold N) = \{ V(0) \}\cup  \bigcup_{1\leq h \leq H}
\{V((n_{h_\circ})_{1 \leq hÐ_\circ \leq h}):
(n_{h_\circ})_{1 \leq hÐ_\circ \leq h}
 \in \prod _{1 \leq hÐ_\circ \leq h} [1, N_{h_\circ}]\},
$$
and edge set $\mathcal F(\bold N)   \cup \mathcal E(\bold N),$ where
$$
\mathcal F(\bold N) =   \bigcup_{1\leq h \leq H} \{f(( n_{h_\circ})_{1 \leq hÐ_\circ \leq h} )   ):(n_{h_\circ})_{1 \leq hÐ_\circ \leq h}
 \in \prod _{1 \leq hÐ_\circ \leq h} [1, N_{h_\circ}]  \},
$$
$$
\mathcal E(\bold N) = \{ e((n_h)_{1 \leq h \leq H + 1}):(n_h)_{1 \leq h \leq H + 1 \in 
 \prod_{1 \leq h \leq H + 1} [1, N_h]}    \},
$$
and with source and target mappings, that are given by
$$
s(f(n_1) = V(0), \qquad 1 \leq  n_1 \leq N_1,
$$
and
\begin{align*}
&s(f(( n_{h_\circ})_{1 \leq hÐ_\circ \leq h} ) ) =
V(  n_{h_\circ})_{1 \leq hÐ_\circ < h}  ),
\\
&t(f(( n_{h_\circ})_{1 \leq hÐ_\circ \leq h} ) ) = V(( n_{h_\circ})_{1 \leq h_\circ \leq h} ) , \quad
(n_{h_\circ})_{1 \leq hÐ_\circ \leq h}
 \in \prod _{1 \leq hÐ_\circ \leq h} [1, N_{h_\circ}]\}, 1 < h \leq H,
 \\
 &s(e(( n_{h})_{1 \leq h \leq H + 1} ) ) =
V((  n_{h})_{1 \leq h \leq H}  ),
\\
&t(e(( n_{h})_{1 \leq h \leq H + 1} ) ) = V(0), \qquad \qquad  \  \  \  \
(n_{h})_{1 \leq h \leq H + 1}
 \in \prod _{1 \leq h \leq H + 1} [1, N_{h}], 
\end{align*}

For given data
$
\bold N = (N_h)_{1\leq h \leq H + 1} \in \bold D_H, H> 1,
$
we denote for $J \in \Bbb N$ by $Y_J(M {\scriptstyle D}(\widehat{G}(\bold N))$ the set of
$y \in M\negthinspace {\scriptstyle D}(\widehat{G}(\bold N))$,
 such that 
$$
\sum_{- J_\circ < j < 0}\psi(y_{j_\circ}) \geq 0, \qquad 0 < J_\circ < J,
$$
and 
$$
\sum_{-J <j <0}\psi(y_j) = -1.
$$
We set
\begin{align*}
Y (M\negthinspace {\scriptstyle D}(\widehat{G}(\bold N))) =  
\bigcap_{j\in \Bbb Z}
S^j_{M\negthinspace {\scriptscriptstyle D}(\widehat{G}(\bold N))}
(\bigcup_{J \in \Bbb N}  
 Y_J (M\negthinspace {\scriptstyle D}(\widehat{G}(\bold N))). \tag{2.1}
\end{align*}
For given data
$
\bold N = (N_h)_{1\leq h \leq H + 1} \in \bold D_H, H> 1,
$
we also define an $\Bbb N$-matrix 
$
A_{h, h^\prime}(\bold N)_{1\leq h, h^\prime  \leq H + 1}
$
with positive entries given by
\begin{align*}
A_{h, h + 1}(\bold N) = N_{h }, \  A_{h + 1, h}(\bold N) = 1, \qquad 1 \leq h \leq H,
\end{align*}
and
\begin{align*}
A_{H + 1, 1} (\bold N)= N_{H+ 1}, \  \  \  \  A_{1, H + 1}(\bold N) = 1.
\end{align*}
We also set
$$
\mathcal C^-(\bold N) =
\{c^-_h(n_h): 1 \leq n_h \leq N_h, 1 \leq h \leq H + 1\}, \ \
\mathcal C^+(\bold N) = \{c^+_h: 1 \leq h \leq H + 1 \},
$$
and we use as a directed graph with adjacency matrix 
$ A(\bold N) $ a directed graph $\overline{G}(\bold N)$ with vertex set
$[1, H + 1]$, with edge set
$\mathcal C^-(\bold N) \cup \mathcal C^+(\bold N)$, and with source and target mappings, that are given by
\begin{align*}
&s( c^-(h, n_h))= h , \qquad \qquad  \ \ \ \ \ 1 \leq n_h \leq N_h,
\qquad \ \ \ 1 \leq h \leq H + 1, 
\\
&t( c^-(h, n_h))= h + 1, \qquad \qquad  \thinspace  1 \leq n_h \leq N_h,
\qquad  \ \ \ 1 \leq h \leq H,
\\
&t( c^-(H + 1, n_{H + 1}))= 1, \qquad  \ \  1 \leq n_{H + 1} \leq N_{H + 1},
\end{align*}
\begin{align*}
&s(c^+(h)) = h + 1, \qquad \qquad 1 \leq h < H + 1,
\\
&t(c^+(h + 1)) = h, \qquad \qquad   1 \leq h \leq H,
\\
&t(c^+(1)) = H + 1.
\end{align*}
We set
\begin{align*}
&\varphi(c^-_h(n_h)) ) = 1, \ 
 \ 1 \leq n_h \leq N_h,
\\
&\varphi(c^+_h) = - 1, \ \qquad 1 \leq h \leq H + 1,
\end{align*}
and
we denote for $I \in \Bbb N$ by 
$X_I(M\negthinspace {\scriptstyle E}(\overline{G}(\bold N)))$ the set of
$x \in M\negthinspace {\scriptstyle E}(\overline{G}(\bold N))$, such that 
$$
\sum_{- I_\circ <i < 0}\varphi(y_{i_\circ}) \geq 0, \qquad 0 < I_\circ < I,
$$
and 
$$
\sum_{-I <i <0}\varphi(y_j) = -1.
$$
We set
\begin{align*}
X (M\negthinspace {\scriptstyle E}(\overline{G}(\bold N))) =  
\bigcap_{i\in \Bbb Z}S^i_{M\negthinspace {\scriptscriptstyle E}(\overline{G}(\bold N))}
(\bigcup_{I \in \Bbb N}   X_I (M\negthinspace {\scriptstyle E}(\overline{G}(\bold N))). \tag{2.2}
\end{align*}

\begin{lemma} 
For $H > 1$, and for data 
$
\bold N = (N_h)_{1\leq h \leq H + 1} \in \bold D_H, H> 1,
$
the Borel dynamical system 
$$
(X (M\negthinspace {\scriptstyle E}(\overline{G}(\bold N)), S_{M \negthinspace{\scriptscriptstyle E}(\bar{G}(\bold N))}
\restriction  X (M \negthinspace{\scriptstyle E}(\overline{G}(\bold N)))
$$
is Borel conjugate to the Borel dynamical system 
$$
(Y (M\negthinspace {\scriptstyle D}(\widehat{G}(\bold N)),
S_{Y (M \negthinspace{\scriptscriptstyle D}(\widehat{G}(\bold N)))}
\restriction  Y (M\negthinspace {\scriptstyle D}(\widehat{G}(\bold N))).
 $$
\end{lemma}
\begin{proof} 
We set
\begin{align*}
&\Omega(f^-(( n_{h_\circ})_{1 \leq h_\circ \leq h}) = c^-_h(n_h), 
\\
&\Omega(f^+(( n_{n_\circ})_{1 \leq h_\circ \leq h} )= c^+_h(n_h), \qquad 1 \leq h \leq H,
\end{align*}
and
\begin{align*}
\Omega(e^-(( n_h)_{1 \leq h \leq H + 1}) = c^-_{H + 1}(n_{H + 1}), \quad
\Omega(e^+(( n_h)_{1 \leq h \leq H + 1}) = c^+_{H + 1}.
\end{align*}
The 1-block map $\Omega$ yields by 
$$
(\omega(y))_0 = \Omega(y_0), \qquad y \in Y (M \negthinspace{\scriptstyle D}(\widehat{G}(\bold N)),
$$
a homomorphism
\begin{multline*}
\omega:
(Y (M \negthinspace{\scriptstyle D}(\widehat{G}(\bold N)),
S_{M\negthinspace {\scriptscriptstyle D}(\widehat{G}(\bold N))}
\restriction  Y (M\negthinspace {\scriptstyle D}(\widehat{G}(\bold N))
 \to 
\\
(X (M\negthinspace {\scriptstyle E}(\overline{G}(\bold N)), S_{M \negthinspace{\scriptscriptstyle E}(\bar{G}(\bold N))}
\restriction  X (M\negthinspace {\scriptstyle E}(\overline{G}(\bold N))).
\end{multline*}
The map $\omega$ is in fact a Borel conjugacy. We indicate how for 
$y\in Y\negthinspace (M {\scriptstyle D}(\widehat{G}(\bold N))$, $y_0$ is reconstructed from 
$x = \omega(y)$, or, more generally, how for $x \in X (M\negthinspace {\scriptstyle E}(\overline{G}(\bold N)), $ one can find $\omega^{-1}(x)_0$.

For this we set inductively
$$
I_1(x) = I, \qquad x \in X_I\in M \negthinspace{\scriptscriptstyle E}(\overline{G}(\bold N)),\  I \in \Bbb N,
$$
$$
I_{k + 1}(x) = I_{k}(x) + I_1(S_{M\negthinspace {\scriptscriptstyle E}(\overline{G}{G}(\bold N)}^{-I_k(x)}(x)),
\qquad x \in M \negthinspace{\scriptstyle E}(\overline{G}(\bold N), \ k \in \Bbb N.
$$
We note that
$$
\varphi(x_{- I_k(x)}) = 1, \ \ \sum_{- I_k(x) < i < 0}\varphi(x_i) = k - 1,\qquad k \in \Bbb N.
$$

We distinguish six cases. In the case that
$$
x_0 \in \{ c^-_1(n_1): 1 \leq n_1 \leq N_1 \},
$$
one has $n_1 \in [1, N_1]$ determined by 
$$
x_0 = c^-_1(n_1),
$$
and one has 
$$
y_0 = f^-(n_1).
$$

In the case that
$$
x_0 = c^+_{H + 1},
$$
one has $(n_h)_{1 \leq h \leq H + 1} \in \prod_{1 \leq h \leq H + 1} [1, N_h]$ determined by
$$
x_{-I_k}= c^-_h(n_h), \qquad 1 \leq h \leq H + 1,
 $$
 and one has
 $$
 y_0 = e^+((n_h)_{1 \leq h \leq H + 1}).
 $$
 
 In the case, that 
 $$
 x_0   \in \{c^-_h(n_h): 1 \leq  h \leq N_h, 1 < h  \leq H\},
 $$
 one has 
 
 $(n_{h_\circ})_{1 \leq h_\circ  < h} \in \prod_{1  \leq h_\circ < h} [1, N_h]$
  determined by
$$
x_{-I_{h_\circ - 1}}= c^-_h(n_{h_\circ}), \qquad 1 \leq h_\circ \leq h,
 $$
 and  $n_h \in [1, N_h]$ determined by 
 $$
 x_0 = c^-_h(n_h),
 $$
 and one has
 $$
 y_0 = f^-( (n_{h_\circ})_{1 \leq h_\circ \leq h}   ).
 $$
In the case that 
$$
x_0 \in \{c^-_{H + 1}: 1\leq n_{H + 1} \leq  N _{H + 1}\},
$$
one has
$$
(n_h)_{1 \leq h \leq H}\in \prod_{1 \leq h \leq H}[1, M_h]
$$
determined by
$$
x_{-I_{h - 1}} = c^-(n_k), \qquad 1 \leq h \leq H,
$$
and 
$
n_{H + 1}\in [1, N_{H + 1}]
$
determined by
$$
x_0 = c^-_{H + 1}
$$
and one has
$$
y_0 = e^-((n_h):{1 \leq h \leq H + 1}).
$$

In the case, that
$$
x_0 \in \{c^+_h :1 \leq h \leq H\},
$$
one has 
$$
(n_{h_\circ})_{1 \leq h_\circ \leq h} \in \prod_{1 \leq h_\circ \leq h}[1, N_{h_\circ}]
$$
determined by
$$
x_{I_{h - h_\circ + 1}} = c^-(n_{h_\circ}), \qquad 1 \leq h_\circ \leq h,
$$
and one has
$$
y_0 = f^-((n_{h_\circ})_{1\leq h_\circ \leq h }).
$$

In the case, that
$$
x_0 = c^+_1,
$$
one has $n_1 \in [1, N_1]$ determined by
$$
x_{-I:1(x)} = c^-_1(n_1),
$$
and one has
$$
y^-_0 = f^-(n_1). \qed
$$
\renewcommand{\qedsymbol}{}
\end{proof}

\begin{theorem}
For $H > 1$ and for data
$$
\bold N = (N_h)_{1\leq h \leq H + 1} \in \bold D_H,
$$
the Markov-Dyck shift of $\widehat G(\bold N)$ has two measures of maximal entropy, and its topological entropy  is equal to the topological entropy of the edge shift of 
$\overline{G}(\bold N)$.
\end{theorem}
\begin{proof}
Let $\mu_{\bold N}$ denote the measure of maximal entropy of 
$M\negthinspace{\scriptstyle E}(\bar G(\bold N))$. It follows from $N_{H + 1}> 1$, that
$$
\mu_{\bold N}(\{ x \in  M\negthinspace {\scriptstyle E}(\overline{G}(\bold N))): \varphi(x_ 0 = 1\}) >
\mu_{\bold N}(\{ x \in  M\negthinspace {\scriptstyle E}(\overline{G}(\bold N)): \varphi(x_ 0 = - 1\}), 
$$
and therefore
$$
\mu_{\bold N}( X(M\negthinspace {\scriptstyle E}(\overline{G}(\bold N))) \geq
\mu_{\bold N}(\{x \in M\negthinspace {\scriptstyle E}(\overline{G}(\bold N)):
\lim_{I \to \infty}\sum_{1 \leq i \leq I}\varphi(x_i) = \infty\}) = 1.
$$
The canonical time reversal of $M\negthinspace {\scriptstyle D}(\widehat{G}(\bold N))$ carries the Borel set 
$Y^-(M\negthinspace {\scriptstyle D}(\widehat{G}(\bold N))$ into  a Borel set 
$Y^+(M\negthinspace {\scriptstyle D}(\widehat{G}(\bold N))$, that is symmetric to
$Y^-(M\negthinspace {\scriptstyle D}(\widehat{G}(\bold N))$.
As a consequence of the Poincar\'e recurrence theorem every ergodic shift invariant 
probability measure on $M\negthinspace {\scriptstyle D}(\widehat{G}(\bold N))$ assigns measure one to 
$Y^-(M\negthinspace {\scriptstyle D}(\widehat{G}(\bold N)))\cup Y^+
(M\negthinspace {\scriptstyle D}(\widehat{G}(\bold N))))$
(see \cite [Section 4]{Kr1}. Apply Lemma 2.1 to obtain 
a measure of maximal entropy of 
$M\negthinspace {\scriptstyle D}(\widehat{G}(\bold N))$ that assigns measure one to $Y^-(M\negthinspace {\scriptstyle D}(\widehat{G}(\bold N)))$.
 The image of this measure under
 the canonical time reversal of $M\negthinspace {\scriptstyle D}(\widehat{G}(\bold N))$ is a measure of maximal entropy of 
$M\negthinspace {\scriptstyle D}(\widehat{G}(\bold N))$, that assigns measure one to 
$Y^+(M\negthinspace {\scriptstyle D}(\widehat{G}(\bold N)))$. By Lemma 2.1 there are no other measures of maximal entropy.
\end{proof}

For data
$\bold N = (N_h)_{1\leq h \leq H + 1} \in \bold D_H, H> 1,$
set
$$
\Sigma(\bold N)  = \sum_{1 \leq h \leq H + 1}N_h, \qquad
 \Pi (\bold N)  = \prod_{1 \leq h \leq H + 1}N_h.
$$

For low heights the dependence of the topological entropy of the shifts 
$\widehat{G}(\bold N)$ on the data $\bold N$ can be made visible.  
We list the Perron eigenvalues of $A(\bold N)$ for heights  2, 3, 5 and 7.  The topological entropy of 
$M\negthinspace {\scriptstyle D}(\widehat{G}(\bold N))$
can then be found by Theorem (2.2) (see (1.1)). 

Consider the case $H(G) = 2$.
Let
$$
\bold N_3 = (N_h)_{1 \leq h \leq 3}) \in\bold D_2. 
$$
The characteristic polynomial of $A(\bold N_3)$ is given by
$$
\Det (z\bold 1 - A(\bold N_3)) = 
- z^3 - \Sigma (\bold N_3)z + 1 + \Pi(\bold N_3).
$$
Set
$$
\Delta(\bold N_3) = (1 + \Pi(\bold N_3))^2  - \tfrac{4}{27}
\Sigma (\bold N_3)^3.
$$

In the case that $\Delta(\bold N_3) > 0$
the Perron eigenvalue of $\lambda (A(\bold N_3 ))$ is given by
\begin{align*}
\lambda (A(\bold N_3 )) = \sqrt[3]{1 + \Pi(\bold N_3) + \sqrt{\Delta(\bold N_3 )}}+
  \sqrt[3]{1 + \Pi(\bold N_3) - \sqrt{\Delta(\bold N_3 )}}. \tag {2.3}
\end{align*}

In the case that $D(\bold N_3) < 0$, the Perron eigenvalue of $\lambda (A(\bold N_3 ))$ is given by
\begin{align*}
\Phi (\bold N_3)= \arccos \left(\frac{3(1 + \Pi(\bold N_3 )}{2\Sigma(\bold N_3 )\sqrt{\Sigma(\bold N_3 )}}  \right), \quad 0 \leq \Phi < \frac {\pi}{2}, \tag{2.4a}
\end{align*}
\begin{align*}
\lambda (A(\bold N_3 ))) =  \cos(\tfrac{1}{3}\Phi(\bold N_3) ). \tag{2.4b}
\end{align*}

Consider the case $H(G) = 3$. Let
$$
\bold N_4 = (N_h)_{1 \leq h \leq 4}) \in \bold D_3. 
$$
Set
$$
\Gamma_0(\bold N_4) = N_1N_3 + N_2N_4.
$$
The characteristic polynomial of $A(\bold N_4)$ is given by
$$
\Det (z\bold 1 - A(\bold N_4)) = 
z^4 - S(\bold N_4)z^2 - 1 + \Gamma_0(\bold N_4) - \Pi(\bold N_4).
$$
Set
$$
\Delta(\bold N_4) = S(\bold N_4)^2 + \Pi(\bold N_4) + 1 - \Gamma_0(\bold N_4) .
$$

The Perron eigenvalue of $A(\bold N_4)$ is given by
\begin{align*}
\lambda(A(\bold N_4))  = \sqrt{\tfrac{1}{2}(\Sigma (\bold N_4 )+ 
\sqrt{\Delta(\bold N_4 )})}. \tag{2.5}
\end{align*}

Consider the case $H(G) = 5$. Let
$$
\bold N_6 = (N_h)_{1 \leq h \leq 6}) \in  \bold D_5. 
$$
Set
$$
\Gamma_0(\bold N_6) =
N_1N_3N_5 + N_2N_4N_6,
$$
$$
\Gamma_1(\bold N_6) = 
N_1N_3 + N_1N_4 + N_1N_5 + N_2N_4 + N_2N_5 +N_2N_6 +
$$
$$
 N_3N_5 + N_5N_4 + N_4N_6.
$$
The characteristic polynomial of $A(\bold N_6)$ is given by
$$
\Det (z\bold 1 - A(\bold N_6))  = z^6 - (\Sigma(\bold N_6) + \Gamma_1(\bold N_6)) z^3 + \Pi(\bold N_6) - 1 -
 \Gamma_0(\bold N_6).
$$
Set
$$
p(\bold N_6) = \tfrac{1}{3}(3\Gamma_1(\bold N_6) - \Sigma (\bold N_6)^2),
$$
$$
q(\bold N_6) = \tfrac{1}{27}( - 2\Sigma (\bold N_6)^3  + 9\Sigma (\bold N_6)\Gamma_1(\bold N_6)   - 27 \Pi(\bold N_6)  - 27 \Gamma_0(\bold N_6)  - 27), 
$$
and
$$
\Delta(\bold N_6) = q(\bold N_6)^2 + \tfrac{4}{27}p(\bold N_6)^3.
$$

In the case that  $\Delta(\bold N_6) \geq 0$, the Perron eigenvalue of $A(\bold N_6)$   is given by
\begin{align*}
&\lambda(A(\bold N_6)) = 
\\
&\sqrt {\tfrac{1}{3}\Sigma (\bold N_6) + 
\sqrt[3]{\tfrac{1}{2}
(- q(\bold N_6) + \sqrt {\Delta(\bold N_6)}\thinspace  )} + \sqrt[3]{\tfrac{1}{2}
(- q(\bold N_6) - \sqrt {\Delta(\bold N_6)}\thinspace ) }}. \tag{2.6}
\end{align*}

In the case, that $\Delta(\bold N_6) < 0$, 
the Perron eigenvalue of $A(\bold N_6)$ is given by
\begin{align*}
\Phi (\bold N_6)=\arccos(\frac{3q(\bold N_6)}{2p(\bold N_6)}\sqrt{- \tfrac{1}{3}p(\bold N_6)}  \thinspace), \tag{2.7.a}
\end{align*}

\begin{align*}
\lambda(A(\bold N_6)) = 
\sqrt {\tfrac{1}{3}\Sigma (\bold N_6) + 
2\sqrt{- \tfrac{1}{3}p(\bold N_6)} \thinspace\cos(\tfrac{1}{3} \Phi (\bold N_3)  \thinspace)}.
\tag{2.7.b} 
\end{align*}

We consider the case $H(G) = 7$.  Let
$$
\bold N_8 = (N_h)_{1 \leq h \leq 8}) \in  \bold D_7. 
$$
Set
\begin{align*}
\Gamma_0(\bold N_8) =
 &    N_1 N_3N_5 N_7 + N_2 N_4 N_6 N_8,
\\
\Gamma_2(\bold N_8) =
 &    N_1N_3N_5  + N_1 N_3 N_6 + N_1 N_5 N_7 + 
\\
  &   N_1N_4N_6 + N_1N_4N_7 + N_1N_5N_7 +
\\
&      N_2N_4N_6 + N_2N_4N_7 + N_2N_4N_8 +
\\
&     N_2N_5N_7 +  N_2 N_5 N_8 +  N_2N_6N_8 +
 \\ 
  &    N_3 N_5 N_7  + N_3 N_5 N_8 + N_5N_6N_8 +
      \\
      & N_4N_6N_8 + N_3N_6N_8 + N_4N_6N_8,
      \\
\Gamma_4(\bold N_8)  =
&N_1N_3 + N_1N_4 + N_1N_5 + N_1N_6 +  N_2N_4 +
\\
&N_1N_3 + N_1N_4 + N_1N_5 + N_1N_6 + N_1N_7 +
\\
&  N_2N_4 + N_2N_5 + N_2N_6 + N_2N_7  + N_2N_8 +
\\
&  N_3N_5 +  N_3N_6 + N_3N_7  + N_3N_8 +N_4N_6  +  
\\
&N_4N_7 + N_4N_8 + N_5N_7 + N_5N_8 + N_6N_8.
\end{align*}
The characteristic polynomial of $A(\bold N_8)$ is given by
\begin{align*}
&\Det (z\bold 1 - A(\bold N_8))  =
\\
&z^8  -  \Sigma(\bold N_8) z^6 + 
\Gamma_4(\bold N_8)z^4 - \Gamma_2(\bold N_8) z^2 +
 \Gamma_0(\bold N_8) + \Pi(\bold N_8) - 1. \tag {2.8}
\end{align*} 
We set
\begin{align*}
 &a_3(\bold N_8) =  -  \Sigma(\bold N_8), 
\ \ a_2(\bold N_8) = \Gamma_4(\bold N_8),
\\
&a_1(\bold N_8) = - \Gamma_2(\bold N_8),
 \ a_0(\bold N_8) =  \Gamma_0(\bold N_8) + \Pi(\bold N_8) - 1.
\end{align*}
With the variable 
$$
y = z^2,
$$
we obtain from (2.8)  the quartic equation 
\begin{align*}
y^4  + a_3(\bold N_8)y^3 + a_2(\bold N_8)y^2 + a_1(\bold N_8) y + a_0(\bold N_8) = 0. \tag {2.9}
\end{align*}
With the coefficients 
$$
p(\bold N_8) = \tfrac{1}{8}(   8a_2(\bold N_8) - 3a_3(\bold N_8)^2),
$$
$$
q(\bold N_8) =\tfrac{1}{8}(a_3(\bold N_8) ^3 -4a_3(\bold N_8)a_2(\bold N_8)   + 8 a_1(\bold N_8)  ),
$$
$$
r(\bold N_8)  = \tfrac{1}{256}( -4a_3(\bold N_8)^4   -64a_3(\bold N_8)a_1(\bold N_8)  + 16 a_3(\bold N_8)^2a_2(\bold N_8)   ) +  a_0(\bold N_8) ,
$$
and have with the variable
$$
x = y - \tfrac{1}{3}a_3(\bold N_8),
$$
the  depressed version
\begin{align*}
x^4 +  p(\bold N_8)   x^2 +  q(\bold N_8) x  +r(\bold N_8)  = 0. \tag{2.10}
\end{align*}
of (2.8).
Set
$$
A_0(\bold N_8) = - q(\bold N_8)^2, \
A_1(\bold N_8) = p(\bold N_8)^2 - 4r(\bold N_8), \
A_2(\bold N_8) =2p(\bold N_8).
$$
A  Descartes resolvent (see, for instance, \cite[Section 6.2]{T}, also see
\cite[Section 59]{V}) of the depressed version (2.10) is  given by
\begin{align*}
u^3 + A_2(\bold N_8)u^2 + A_1(\bold N_8)u + A_0(\bold N_8) = 0. \tag {2.11}
\end{align*}
Set
$$
P(\bold N_8) = \tfrac{1}{3}(3 A_1(\bold N_8)  -  A_2(\bold N_8)^2   ),
$$
$$
Q(\bold N_8) =\tfrac{1}{27}( 2  A_2(\bold N_8)^3 - 9A_1(\bold N_8)A_1(\bold N_8)) +
 A_0(\bold N_8) .
$$
The depressed version of the resolvent (2.11) is given by
$$
v = u - \tfrac{1}{3}A_2(\bold N_8),
$$
$$
v^3 + P(\bold N_8)  v + Q(\bold N_8) = 0.
$$
Set
$$
\Delta(\bold N_8) = Q(\bold N_8)^2 + \tfrac{4}{27}P(\bold N_8)^3.
$$

In the case that
$$
\Delta(\bold N_8)> 0,
$$
set
$$
u(\bold N_8) = \tfrac{1}{3}A_2(\bold N_8) + \sqrt[3]{\tfrac{1}{2}(- Q(\bold N_8) + \sqrt{\Delta(\bold N_8)})} +
\sqrt[3]{\tfrac{1}{2}(- Q(\bold N_8) - \sqrt{\Delta(\bold N_8)})} ,
$$
and set
$$
\delta(\bold N_8) = p(\bold N_8)u(\bold N_8)^4 + q(\bold N_8)u(\bold N_8)^3
 - (p(\bold N_8)^2 - 4r(\bold N_8))u(\bold N_8)^2 - 3q(\bold N_8)^2,
$$
The  Perron eigenvalue of $A(\bold N_8)$ is given by
\begin{align*}
\lambda(A(\bold N_8)) =
\begin{cases}
\sqrt{\tfrac{1}{2}  (- \sqrt{u(\bold N_8)} +     \sqrt{ u(\bold N_8) + p(\bold N_8) - \frac{2q(\bold N_8)}{\sqrt{u(\bold N_8)}}} ) }, &\text{if  $\delta(\bold N_8)> 0$}, \\
\sqrt{\tfrac{1}{2}  ( \sqrt{u(\bold N_8)} +     \sqrt{ u(\bold N_8) + p(\bold N_8) + \frac{2q(\bold N_8)}{\sqrt{u(\bold N_8)}}} ) }, &\text {if $\delta(\bold N_8) < 0$}. \tag{2.12}
\end{cases}
\end{align*}

In the case that
$$
\Delta(\bold N_8)> 0,
$$
set
$$
u(\bold N_8) = 2\sqrt{-\frac{P(\bold N_8)}{3}} \cos( \tfrac{1}{3}\Phi(\bold N_8) ).
$$
The Perron eigenvalue of $A(\bold N_8)$ is given by
\begin{align*}
\Phi(\bold N_8) = \arccos
(-\frac{3Q(\bold N_8)}{2P(\bold N_8)\sqrt{-\frac{P(\bold N_8)}{3}}}), \tag{2.13a}
\end{align*}
\begin{align*}
\lambda(A(\bold N_8)) =
\begin{cases}
\sqrt{\tfrac{1}{2}  (- \sqrt{u(\bold N_8)} +  \sqrt{ u(\bold N_8) + p(\bold N_8) - \frac{2q(\bold N_8)}{\sqrt{u(\bold N_8)}}} ) }, &\text{if  $q(\bold N_8)> 0$}, \\
\sqrt{\tfrac{1}{2}  ( \sqrt{u(\bold N_8)} +     \sqrt{ u(\bold N_8) + p(\bold N_8) + \frac{2q(\bold N_8)}{\sqrt{u(\bold N_8)}}} ) }, &\text {if $q(\bold N_8) < 0$}. \tag{2.13b}
\end{cases}
\end{align*}

\begin{proposition} Let $H > 1, L > 1,$  let 
$$
\bold N = (N_h)_{1 \leq h \leq H + 1}) \in  \bold D_H,
$$
and let
$$
\widetilde{\bold N} = ((N_h)_{1 \leq h \leq H + 1})_{0 \leq \ell < L}.
$$
The edge shift of $(\overline{G}(\bold N))$ is the image of the edge shift of 
$(\overline{G}(\widetilde{\bold N}))$
under a 1-bi-resolving homomorphism. 
\end{proposition}
\begin{proof}
We introduce notation by setting
$$
(\widetilde c^-_{\widetilde h})_{1 \leq  \widetilde h \leq L(H + 1)} =
((N_h)_{1 \leq  \widetilde h \leq H + 1})_{1 \leq  \ell \leq L}
$$
$$
(\widetilde c^+_{\widetilde h})_{1 \leq  \widetilde h \leq L(H + 1)} =
((N_h)_{1 \leq  \widetilde h \leq H + 1})_{1 \leq  \ell \leq L}
$$
and
\begin{align*}
&\Theta (\widetilde c^-_{h + \ell(H +1)}(n_h)) = c^-_{h}(n_h),
\\
&\Theta (\widetilde c^-_{h + \ell(H +1)}(n_h)) = c^-_{h}(n_h),
 \qquad 1 \leq h \leq H + 1, 0 \leq \ell < L.
\end{align*}
From the one-block map $\Theta$ we have a homomorphism
$$
\vartheta: M\negthinspace{\scriptstyle E}(\overline{G}(\widetilde {\bold N}))  \to
 M\negthinspace{\scriptstyle E}(\overline{G}(\bold N))
$$
by
$$
\vartheta (\widetilde {x} ) = (\Theta( \widetilde {x}_i  ))_{i \in \Bbb Z}
$$
One shows, that $\vartheta$ is 1-right-resolving.
Let 
$$
\widetilde {x} \in M\negthinspace{\scriptstyle E}(\overline{G}(\bold N)),\quad x = 
\vartheta(\widetilde x),
$$
and for $h \in [1, H + 1], \ell \in [0, L)$, let
$$
t(x_0) = h, \quad t(\widetilde{x}_0) = h - \ell (H + 1).
$$
If 
$$
x_1 = c^-_{h + 1}(n_{h + 1})
$$
then necessarily
$$
\widetilde{x}_1 = \widetilde{c}_{h + 1 + \ell(H + 1) (\mod ( L(H + 1)}(n_{h + 1}),
$$
and if
$$
x_1 = c^+_h,
$$
then necessarily
$$
\widetilde{x}_1 = \widetilde{c}^+_{h + \ell(H + 1)}.
$$
The proof, that $\vartheta$ is 1-left-resolving is similar.
\end{proof}

\begin{proposition}
Let $H > 1$ and let 
$$
\bold N = (N_h)_{1 \leq h \leq H + 1}) \in  \bold D_H. 
$$
The topological entropy of the Markov Dyck shift of $\widehat G(\widetilde {\bold N})$  is equal to the topological entropy of the edge shift of $\overline{G}(\widetilde {\bold N})$.
\end{proposition}
\begin{proof}
Apply Theorem 3.2 and Proposition 3.3.
\end{proof}

For Proposition 2.3 and Proposition 2.4 compare (2.12) and (2.5), (2.12) and (42.5,  (2.6 ) and (2.3), (2.6) and (2.3) and (2.4a - b) and (2.7a - b). 

\section{The case $H(G) =1$}
For given data $(N, M) \in \bold D_1$,
we introduce the matrix
$$
A(N, M) =
\left(\begin{matrix}
0 & N + 1 
\\
M + 1 & 0
\end{matrix}\right).
$$
We set
$$
\mathcal C^-(N, M) = \{c^-_1(n): 1 \leq n \leq N\} \cup \{c^-_2(m): 1 \leq m \leq M\},
$$
$$
\mathcal C^+(N, M) = \{c^+_1, c^+_2\},
$$
and we use as a directed graph with adjacency matrix $A(N, M)$ a graph with vertex set 
$\{1, 2\}$ and edge set $\mathcal C^-(N, M) \cup \mathcal C^+(N, M)$,
and with source and target mappings, that are given by
$$
s(e^-_1(n)) = 2, \  t(c^-_1(n)) = 1, \qquad 1 \leq n \leq N,
$$
$$
s(e^-_2(m)) = 2,  \ t(c^-_2(m)) = 1, \qquad 1 \leq m \leq M,
$$
$$
s(c^+_1  ) = t( c^+_2) = 1, \qquad s(c^+_2) = t( c^+_1 ) = 2.
$$
We set
\begin{align*}
&\varphi (c^-_1(n)) = 1, \qquad  \ 1 \leq n \leq N,
\\
&\varphi (c^-_1(m)) = 1, \qquad 1 \leq m \leq M,
\end{align*}
$$
\varphi (c^+_1) = \varphi ( c^+_2) = -1.
$$

With the directed graphs $\widehat {G} (N, M)$ and $\overline{G} (N, M)$ and the mapping 
$\varphi$ at hand one can define Borel sets 
$X (M {\scriptstyle E}(\bar{G}(N, M))$ 
and 
$M\negthinspace {\scriptstyle D}(\widehat{G}(N, M))$
that are analogous to the sets defined in (2.1) and (2.2). Then one proceeds exactly as in the case $H(G) > 1$. We list the corresponding statements.

\begin{lemma} 
For data 
$(N, M) \in \Bbb N \times (\Bbb N \setminus \{ 1 \}  )$
the Borel dynamical system 
$$
(X (M\negthinspace {\scriptstyle E}(\bar{G}(N, M), S_{M\negthinspace {\scriptscriptstyle E}(\overline{G}(N, M))}
\restriction  X (M\negthinspace {\scriptstyle E}(\overline{G}(N, M)))
$$
is Borel conjugate to the Borel dynamical system 
$$
Y (M\negthinspace {\scriptstyle D}(\widehat{G}(N, M))),
S_{Y (M\negthinspace {\scriptscriptstyle D}(\widehat{G}(N, M))}
\restriction  Y (M\negthinspace {\scriptstyle D}(\widehat{G}(N, M))).
 $$
 \end{lemma} 

\begin{theorem}
For data 
$(N, M) \in  \bold D_1$ the Markov-Dyck shift of
$\widehat{G}(N, M))$ has two measures of maximal entropy, and its
topological entropy  is given by
\begin{align*}
\ent( M\negthinspace {\scriptstyle D}(\widehat{G}(N, M)) ) = \tfrac {1}{2} ( \log (N + 1) + \log (M + 1)). \tag {3.1}
\end{align*}
\end{theorem}

\begin{proposition} Let $H > 1, L > 1,$ let 
$$
\bold N = (N_h)_{1 \leq h \leq H + 1}) \in  \bold D_1,
$$
and let
$$
\widetilde{\bold N} = ((N_h)_{1 \leq h \leq H + 1})_{0 \leq \ell < L}.
$$
The edge shift of $(\bar G(\bold N))$ is the image of the edge shift of 
$(\bar G(\widetilde{\bold N}))$
under a 1-bi-resolving homomorphism. 
\end{proposition}

\begin{proposition}
Let $H > 1$ and let 
$$
\bold N = (N_h)_{1 \leq h \leq H + 1}) \in  \bold D_1. 
$$
The topological entropy of the Markov Dyck shift of $\widehat{G}(\bold N)$  is equal to the topological entropy of the edge shift of $\overline{G}(\bold N)$.
\end{proposition}

For Proposition 3.3 and Proposition 3.4 compare (2.5) and (3.1), and (2.3) and (1.2).

For $N > 1$ let $\widehat{G}(N)$ be the one-vertex graph with $N$ loops, and let 
$\overline{G}(N)$ be the one-vertex graph with $N + 1$ loops. The graphs $\widehat{G}(N)$ and $\overline{G}(N)$  can be viewed (for the case $H(G) = 0$) as analogous to the graphs
 $\widehat{G}((N)_{1 \leq h \leq H + 1})$ and 
 $\bar{G}((N)_{1 \leq h \leq H + 1})$ respectively (for the case $H (G) \in \Bbb N$). This is the case $H(G) = 0$, that was considered in \cite {Kr1}.  As the edge shift of the graph $\bar{G}(N)$, which is the full shift on $N + 1$ symbols, is the image of the edge shift of the graph 
$\bar{G}((N)_{1 \leq h \leq H - 1}), H \in \Bbb N$, under a 1-bi-resolving homomorphism, it follows from (1.2)
that
\begin{align*}
 \ent(M\negthinspace {\scriptstyle D}(\widehat{G}((N)_{1 \leq h \leq H + 1}))) = \log (N + 1), 
 \qquad N > 1, H 
 \in \Bbb Z_+.
 \tag {3.2}
\end{align*}
Compare (3.2) and (2.3), and (3.2) and (2.5).

\section{Zeta functions}

As shown in \cite {Ke}, the zeta function of the Dyck shift is given by
\begin{align*}
\zeta_{M\negthinspace {\scriptscriptstyle D}({G}( N))}(z) = 
\frac{2(1 + \sqrt{1 - 4Nz^2})}{(1 - 2N + \sqrt{1 - 4Nz^2})^2}, \qquad N> 1.  \tag {4.1}
\end{align*}
We consider the case $H(G) \in \Bbb N$. Let there be given data 
$$
\bold N = (N_h)_{1 \leq h \leq H + 1} \in  \bold D_H, \quad H \in \Bbb N.
$$
We associate to each vertex $V \in \mathcal V(\bold N )$
of the graph $\widehat {G}(\bold N )$ the circular code $\mathcal C(V)$ that contains the paths
 $$
 b = (b_i)_{1 \leq i \leq 2I}, 
 $$
 in $\widehat {G}(\bold N)$
 such that
 $$
 s(b_1 ) )= t(b_{2I}) = V,
 $$
 and
 $$
 \sum_{1 \leq j \leq J}\psi(b_j) > 0, \qquad 1 < J <2I,
 $$
 and
 $$
 \sum_{1 \leq i \leq 2I}\psi(b_j) = 0.
 $$
 As a consequence of rotational homogeneity one has that for 
 $h \in [1, H]$  the generating functions 
 of the codes
 $
 \mathcal C(V), V \in \prod_{1 \leq h_\circ \leq h} [1, N_{h_\circ}],
 $
 have a common value, that we denote by $g_h =g_h(\bold N) $.
 For  polynomials
 $$
 P^{(0)}_h,P^{(1)}_h, \ Q^{(0)}_h,Q^{(1)}_h, \qquad 0 \leq h \leq H,
 $$
that are defined inductively by
 $$
 P^{(0)}_H(z) = N_{h + 1}z^2, \quad P^{(1)}_H(z) = 0,
 \qquad
 Q^{(0)}_H(z) = 1, \quad Q^{(1)}_H(z) = 1,
 $$
 and
 \begin{align*}
&P^{(0)}_h(z) = N_{h + 1}z^2 Q^{(0)}_{h + 1}
 \qquad  \  \  \  \  \  \  \  \  \ 
P^{(1)}_h(z) = N_{h + 1}z^2 Q^{(1)}_{h + 1}
 \\
 Q^{(0)}_h&(z) =  Q^{(0)}_{h + 1}(z ) - P^{(0)}_{h + 1}(z ), \qquad 
  Q^{(1)}_h(z) = Q^{(1)}_{h + 1}(z ) - P^{(1)}_{h + 1}(z ),\qquad \ \ 0 \leq h < H,
  \end{align*}
it holds that
\begin{align*}
&g_h = \frac{P^{(0)}_h - P^{(1)}_hg_{h + 1}}{Q^{(0)}_h - Q^{(1)}_hg_{h + 1}}, 
\qquad 0 \leq h <H, \tag {4.2}
\\
&g_H = \frac{P^{(0)}_H - P^{(1)}_Hg_{0}}{Q^{(0)}_H - Q^{(1)}_Hg_{0}}.
\end{align*}
It is seen, that the generating functions $g_h, 0 \leq h \leq H,$ are given by periodic Jacobi continued fractions (see \cite [Section 2]{Ka}).
It follows that
\begin{align*}
g_0 = \frac{1}{2Q^{(1)}_1}\left(P^{(0)}_0 + Q^{(0)}_0 - \sqrt{(P^{(0)}_0 + Q^{(0)}_0)^2
- 4P^{(0)}_0(Q^{(1)}_0)^2}\thinspace\right).
\end{align*}
Corresponding formulae for the generating functions $g_h, 0 < h \leq H,$ can be obtained from (4.2) or by cyclically permuting the data.

\begin{proposition} 
For data
$$
\bold N =(N_h)_{1 \leq h \leq H + 1} \in  \bold D_H, \quad
H \in \Bbb N,
$$
the zeta function of the Markov-Dyck shift of $\widehat G(\bold N )$ is given by
\begin{align*}
&\zeta_{M\negthinspace {\scriptscriptstyle D}(\widehat{G}(\bold N))}= 
\\
&(\prod_{0\leq h \leq H} (1 - g_h(\bold N )) - \Pi(\bold N) z^{2(H + 1)})^{-2}
\prod_{0\leq h \leq H}[(1 - g_h(\bold N ))^{2 -
 \prod_{1 \leq h_\circ \leq h}( N_{h_\circ})}].
\end{align*}
\end{proposition}
\begin{proof} 
The zeta function of the set of  neutral periodic points of
$M\negthinspace {\scriptstyle D}(\widehat{G}(\bold N)))$  is
\begin{align*}
\prod_{0\leq h \leq H}[(1 - g_h)^{ -
 \prod_{1 \leq h_\circ \leq h}( N_{h_\circ})}]. \tag {4.3}
\end{align*}
(see \cite{KM4}). 
For the periodic points of $M\negthinspace{\scriptstyle{D}}(\widehat G(\bold N ))$ with negative multiplier we use the circular code $\mathcal C^-$ of cycles
$$
b = (b_i)_{1 \leq i \leq 2I + H + 1}, \qquad I \in \Bbb Z_+,
$$
in $\mathcal E^-  \cup  \mathcal E^-$ such that 
$$
s(b_1) = t(b_{ 2I + H + 1} ) = V(0),
$$
and such that
$$
\sum_{1\leq j \leq J} \psi(b_j) \geq 0, \qquad 1 \leq i \leq 2I + H + 1,
$$
and
$$
\sum_{1 \leq i \leq 2I + H + 1}\psi(b_i) = H + 1.
$$ 

The generating function of the code $\mathcal C^-$ is given by
\begin{align*}
g_{\mathcal C} = z^{2(H + 1)}\prod_{0\leq h \leq H} N_{H + 1}(1 - g_h)^{-1}.  \tag{4.4}
\end{align*}
The canonical time reversal of 
$M{\scriptstyle{D}}(\widehat G(\bold N))$ carries  periodic points with a negative multiplier into periodic points with a positive multiplier and vice versa. Apply (4.3) and (4.4).
\end{proof}

\begin{corollary} 
For data $(N, M) \in  \bold D_1$ set
$$
F(N, M)(z) = 
\tfrac{1}{2}\left(1 - (N - M)z^2 - \sqrt{(1 + (N - M)z^2)^2- 4Nz^2 } \thinspace \right).
$$
The zeta function of $M\negthinspace {\scriptstyle D}(\widehat G(N, M))$ is given by
$$
\zeta_{M\negthinspace {\scriptscriptstyle D}(\widehat G(N, M))} (z)= 
\frac{F(N,M)F(M,N)}
{F(N,M)^N(F(N,M)F(M,N)- NMz^2)}.
$$
\end{corollary} 
\begin{proof} 
The corollary follows from Proposition (4.1).
\end{proof}
\noindent

As is seen from (4.2), in the case of periodic data 
$((N_h)_{1 \leq h \leq H + 1})_{0 \leq \ell < L}$,
the sequence of generating functions $(g_{h + \ell(h + 1})_{1 \leq h \leq H + 1, 0 \leq \ell < L}$ has the same period $L$ as the data. 

\begin{proposition} 
For data
$$
\bold N =((N_h)_{1 \leq h \leq H + 1})_{0 \leq \ell < L} \in 
 \bold D_H, \quad
H \in \Bbb N,
$$
the zeta function of the Markov-Dyck shift of $\widehat G(\bold N )$ is given by
\begin{align*}
&\zeta_{M\negthinspace {\scriptscriptstyle D}(\widehat{G}(\bold N))}= 
\\
&((\prod_{0\leq h \leq H} (1 - g_h(\bold N ))^L) - \Pi(\bold N)^L z^{2(H + 1)})^{-2}
\prod_{0\leq h \leq H}[(1 - g_h(\bold N ))^{2 -
 \prod_{1 \leq h_\circ \leq h}( N_{h_\circ})}].
\end{align*}
\end{proposition}

Proposition 4.3 together with (4.1) gives
\begin{multline*}
\zeta_{M{\negthinspace\scriptscriptstyle{D}}(\widehat G((N)_{1 \leq n \leq H + 1}))}(z)
= 
\\
2^{\frac{N{H + 1} - 1}{N - 1}}(1 + \sqrt{1 - 4Nz^2})^{2H + 2 - \frac{N^{H + 1} - 1}{N - 1}}
((1 +  \sqrt{1 - 4Nz^2} )^{H + 1 } - (2^Nz)^{H + 1})^{-2}, 
\\
 H \in \Bbb Z_+.
\end{multline*}

\bigskip

\par\noindent Wolfgang Krieger
\par\noindent Institute for Applied Mathematics, 
\par\noindent  University of Heidelberg,
\par\noindent Im Neuenheimer Feld 205, 
 \par\noindent 69120 Heidelberg,
 \par\noindent Germany
\par\noindent krieger@math.uni-heidelberg.de

\end{document}